\documentclass[11pt]{article}
\usepackage{fullpage}
\usepackage{microtype}
\usepackage{xfrac}
\usepackage{color}
\usepackage{amsmath,amsthm,amssymb,amsfonts}
\usepackage[numbers]{natbib}

\usepackage[]{color-edits}

\usepackage{amsmath,graphicx}

\usepackage{array}
\usepackage[outline]{contour}
\usepackage{amsthm}
\usepackage{subfig,placeins} 
\usepackage{graphicx,color,float,epstopdf}
\usepackage{fix-cm}
\usepackage{amssymb}
\usepackage{arydshln}
\usepackage{tabularx}
\usepackage{mathtools}
\usepackage{graphicx}
\usepackage{psfrag}
\usepackage{epsfig}
\usepackage{graphics,bm,paralist,ifthen,color, algorithm}
\usepackage{array}
\usepackage{algpseudocode}
\usepackage{calc}
\usepackage{longtable}
\usepackage{mathrsfs}
\usepackage{float}
\usepackage{lipsum}
\usepackage{thmtools,thm-restate}
\usepackage{lmodern}
\usepackage[utf8]{inputenc}
\usepackage{thmtools}
\usepackage{thm-restate}
\usepackage{hyperref}
\usepackage{amsthm}

\usepackage{cleveref}
\def\BState{\State\hskip-\ALG@thistlm}
\makeatother

\algnewcommand{\IfThenElse}[3]{
	\State \algorithmicif\ #1\ \algorithmicthen\ #2\ \algorithmicelse\#3}

\newcommand*\circled[1]{\tikz[baseline=(char.base)]{
            \node[shape=circle,draw,inner sep=2pt] (char) {#1};}}

\DeclareMathOperator*{\argmin}{arg\,min}

\newcommand{\bd}{{\mathbf{d}}}
\newcommand{\bx}{{\mathbf{x}}}
\newcommand{\by}{{\mathbf{y}}}

\newcommand{\bt}{\btheta_{t}, \balpha_{t+1}}

\newcommand{\bz}{{\mathbf{z}}}
\newcommand{\cZ}{{\mathcal{Z}}}

\newcommand{\cSt}{\Theta}
\newcommand{\cSa}{\mathcal{A}}
\newcommand{\cX}{{\cal X}}
\newcommand{\cY}{{\cal Y}}

\newcommand{\gl}{g}
\newcommand{\ql}{q}

\newcommand{\hl}{\nabla_{\btheta} h(\btheta_t, \balpha_{t+1})}
\newcommand{\hls}{\nabla_{\btheta} h(\btheta_t, \balpha^*(\btheta_t))}
\newcommand{\hlss}{\nabla_{\balpha} h(\btheta_t, \balpha^*(\btheta_t))}
\newcommand{\hla}{\nabla_{\balpha} h(\btheta_t, \balpha_{t+1})}

\newcommand{\btheta}{\boldsymbol{\theta}}
\newcommand{\balpha}{\boldsymbol{\alpha}}

\usepackage{tikz}
\usetikzlibrary{decorations.pathreplacing,calc}
\newcommand{\tikzmark}[1]{\tikz[overlay,remember picture] \node (#1) {};}

\newcommand*{\AddNote}[4]{%
    \begin{tikzpicture}[overlay, remember picture]
        \draw [decoration={brace,amplitude=0.7em},decorate,thin,black]
            ($(#3)!(#1.north)!($(#3)-(0,1)$)$) --  
            ($(#3)!(#2.south)!($(#3)-(0,1)$)$)
                node [align=center, text width=2.5cm, pos=0.5, anchor=west] {#4};
    \end{tikzpicture}
}%

\usepackage{algpseudocode}

\newtheorem{asu}{Assumption}

\newtheorem{defi}{Definition}
\newtheorem{rmk}{Remark}
\newtheorem{thm}{Theorem}

\providecommand{\keywords}[1]
{
  \small	
  \textbf{\textit{Keywords---}} #1
}

\addauthor{mr}{blue}

\title{Solving Non-Convex Non-Differentiable Min-Max Games using Proximal Gradient Method\thanks{This arXiv submission includes the details of the proofs for the paper accepted for publication in the proceeding of the $45^{th}$ International Conference on Acoustics, Speech, and Signal Processing (ICASSP).}}

\newcommand*{\email}[1]{\texttt{#1}}
\author{
Babak Barazandeh
\and 
Meisam Razaviyayn}
\date{
University of Southern California\\
\email{\{barazand,razaviya\}@usc.edu}
}

\begin{document}
%
\maketitle
\begin{abstract}
Min-max saddle point games appear in a wide range of applications in machine leaning and signal processing. Despite their wide applicability, theoretical studies are mostly limited to the special convex-concave structure. While some recent works generalized these results to special smooth non-convex cases, our understanding of non-smooth scenarios is still limited. In this work, we study special form of non-smooth min-max games when the objective function is (strongly) convex with respect to one of the player's decision variable. We show that a simple multi-step proximal gradient descent-ascent algorithm converges to $\epsilon$-first-order Nash equilibrium of the min-max game with the number of gradient evaluations being polynomial in $1/\epsilon$. We will also show that our notion of stationarity is stronger than existing ones in the literature. Finally, we evaluate the performance of the proposed algorithm through  adversarial attack on a LASSO estimator.
\end{abstract}

\section{Introduction}
Non-convex min-max  saddle point games appear in a wide range of applications such as training Generative Adversarial Networks~\cite{goodfellow2014generative,gulrajani2017improved,sanjabi2018convergence, barazandeh2019training}, fair statistical inference \cite{xu2018fairgan, madras2018learning, baharlouei2019r}, and training robust neural networks and systems \cite{madry2017towards, berger2013statistical,barazandeh2018behavior}. In such a game, the goal is to solve the optimization problem of the form
\begin{align}\label{sec:intro}
\min_{\btheta \in \cSt}\; \max_{\balpha \in \cSa} \;\;f(\btheta,\balpha),
\end{align}
which can be considered as a two player game where one player aims at increasing the objective, while the  other tries to minimize the objective. Using game theoretic point of view, we may aim for finding Nash equilibria~\cite{nash1950equilibrium} in which  no player can do better off by unilaterally changing its strategy.  Unfortunately, finding/checking such Nash equilibria is hard in general~\cite{murty1987some} for non-convex objective functions. Moreover, such Nash equilibria might not even exist. 
Therefore, many works focus on special cases such as  convex-concave problems where $f(\btheta,.)$ is  concave for any given $\btheta$ and $f(.,\balpha)$ is  convex for any given $\balpha$. Under this assumption, different algorithms  such as optimistic mirror descent \cite{rakhlin2013optimization, mertikopoulos2019optimistic, daskalakis2018last,mokhtari2019unified},  Frank-Wolfe algorithm \cite{gidel2016frank,abernethy2017frank} and Primal-Dual method~\cite{hamedani2018iteration} have been studied.

\vspace{0.2cm}

In the general non-convex settings, \cite{rafique2018non} considers the weakly convex-concave case and proposes a primal-dual based approach for finding approximate stationary solutions. 
More recently, the research works~\cite{lu2019block, lu2019hybrid, nouiehed2019solving, ostrovskii2020efficient} examine the min-max problem in non-convex-(strongly)-concave cases and proposed first-order algorithms for solving them. 
 Some of the results have been accelerated in the ``Moreau envelope regime" by the recent interesting work~\cite{thekumparampil2019efficient}. This work first starts by studying the problem in smooth strongly convex-concave and convex-concave settings, and proposes an algorithm based on the combination of Mirror-Prox \cite{juditsky2011solving} and Nesterov's accelerated gradient descent \cite{nesterov1998introductory} methods. Then the algorithm is extended  to the smooth non-convex-concave scenario.
Some of the aforementioned results are extended to zeroth-order methods for solving non-convex-concave min-max optimization problems~\cite{liu2019min,wang2020zeroth}. 
As a first step toward solving non-convex non-concave min-max problems, \cite{nouiehed2019solving} studies a  class of games in which one of the players satisfies the Polyak-\L{}ojasiewic(PL) condition and the other player has a general non-convex structure. More recently, the work~\cite{yang2020global} studied the two sided PL min-max games and proposed a variance reduced strategy for solving these games.

\vspace{0.2cm}


While almost all existing efforts focus on smooth min-max problems, in this work, we study  non-differentiable, non-convex-strongly-concave and non-convex-concave games and propose an  algorithm for computing their first-order Nash equilibria. 
\section{Problem Definition}

Consider the min-max zero-sum game 
\begin{align}
\min_{\btheta \in \cSt}\; \max_{\balpha \in \cSa} \;\;(f(\btheta,\balpha) \triangleq h(\btheta,\balpha) -p(\balpha) + q(\btheta)),\label{eq: game1-cons}
\end{align}
where we assume that the constraint sets and the objective function satisfy the following assumptions throughout the paper.
\begin{asu}
\label{assumption:Concavity}
The sets $\cSt \subseteq \mathbb{R}^{d_\theta}$ and $\cSa \subseteq \mathbb{R}^{d_\alpha}$ are convex and compact. Moreover, there exist two separate balls with radius $R$ that contains the feasible sets~$\cSa$ and~$\cSt$. 
\end{asu}
\begin{asu}
\label{assumption:objective}
The functions $h(\btheta,\balpha)$ is continuously differentiable, $p(\cdot)$ and $q(\cdot)$ are convex and (potentially) non-differentiable, $p(\cdot)$ is $L_{p}$-Lipschitz continuous and $q(\cdot)$ is continuous.  
\end{asu}


\begin{asu}
\label{assumption: LipSmooth-uncons} 
	The function~$h(\btheta, \balpha)$ is continuously differentiable in both $\btheta$ and $\balpha$ and there exist constants $L_{11}$, $L_{22}$ and $L_{12}$ such that for every $\balpha,\balpha_1,\balpha_{2}\in \cSa$, and $\btheta,\btheta_1,\btheta_2 \in \cSt$, we have
	\[	\begin{array}{ll}
	&\|\nabla_{\btheta} h(\btheta_1,\balpha)-\nabla_{\btheta} h(\btheta_2,\balpha)\|\leq L_{11}\|\btheta_1-\btheta_2\|,\nonumber\\
	& \|\nabla_{\balpha} h(\btheta,\balpha_1)-\nabla_{\balpha} h(\btheta,\balpha_{2})\|\leq L_{22}\|\balpha_1-\balpha_{2}\|,\\
	&\|\nabla_{\balpha} h(\btheta_1,\balpha)-\nabla_{\balpha} h(\btheta_2,\balpha)\|\leq L_{12}\|\btheta_1-\btheta_2\|,\nonumber\\
	& \|\nabla_{\btheta} h(\btheta,\balpha_1)-\nabla_{\btheta} h(\btheta,\balpha_{2})\|\leq L_{12}\|\balpha_1-\balpha_{2}\|.\\
	\end{array}  \]
	\normalsize
\end{asu}

\vspace{0.2cm}

To proceed, let us first define some preliminary concepts:

\begin{restatable}{defi}{df}{\normalfont (Directional Derivative)}\label{def.dd}
	Let $\psi: \mathbb{R}^n \rightarrow \mathbb{R}$ and $\bar{\bx} \in dom(\psi)$. The directional derivative of $\psi$ at the point $\bar{\bx}$ along the direction $\bd$ is defined as
	\begin{align*}
	\psi^{'}({\bar{\bx}};\bd) = \lim\limits_{\tau \downarrow 0} \frac{\psi({\bar{\bx}} + \tau \bd)  - \psi({\bar{\bx}})}{\tau}.
	\end{align*}
	We say that $\psi$ is directionally differentiable at $\bar{\bx}$ if the above limit exists for all $\bd \in \mathbb{R}^n$. It can be shown that any convex function is directionally differentiable.     
\end{restatable}

\begin{restatable}{defi}{dff}{\normalfont (FNE)}
	\label{def.FNE}
	A point $(\btheta^*,\balpha^*)\in \cSt\times \cSa$ is a \textit{first-order Nash equilibrium} (FNE) of the game~\eqref{eq: game1-cons} if  
	\begin{align*}
	    &f'_{\btheta}(\btheta^*, \balpha^*; \btheta - \btheta^*)\geq 0 \quad \forall \btheta \in \cSt, \\
	    & f'_{\balpha}(\btheta^*, \balpha^*; \balpha - \balpha^*)\leq 0 \quad \forall \balpha \in \cSa;
	\end{align*}
	or equivalently if 
\small
\begin{align}\small
	&\langle \nabla_{\btheta} h(\btheta^*, \balpha^*), \btheta -\btheta^* \rangle + q(\btheta) - q(\btheta^*)+ \frac{M}{2}||\btheta - \btheta^*||^2   \geq 0,\nonumber\\
&\langle \nabla_{\balpha} h(\btheta^*, \balpha^*), \balpha -\balpha^* \rangle  - p(\balpha) + p(\balpha^*) - \frac{M}{2} ||\balpha - \balpha^*||^2 \leq 0, \nonumber
\end{align}\normalsize
for all $ \btheta \in \cSt$ and $\balpha \in \cSa$; and all $M>0$.
\end{restatable}
\normalsize

This definition implies that, at the first-order Nash equilibrium point, each player satisfies the first-order necessary optimality condition of its own objective when the other player's strategy is fixed. This is also equivalent to saying we have found the solution to the corresponding variational inequality \cite{harker1990finite}.
Moreover, in the unconstrained  smooth case that $\cSt = \mathbb{R}^{d_\theta}$, $\cSa =\mathbb{R}^{d_\alpha}$, and $p\equiv q \equiv 0$, this definition reduces to the standard widely used definition~$\nabla_{\balpha} h(\btheta^*,\balpha^*) = 0$ and $\nabla_{\btheta} h(\btheta^*,\balpha^*) = 0$.  

\vspace{0.2cm}

In practice,  we use iterative methods  for solving such games and it is natural to evaluate the performance of the algorithms based on their efficiency in finding an approximate-FNE point. To this end, let us  define the concept of approximate-FNE point: 
 
\begin{restatable}{defi}{dfff}(Approximate-FNE)\label{def:cons-approx-stationarity}
	A point $(\bar{\btheta}, \bar{\balpha})$ is said to be an $\epsilon$--first-order Nash equilibrium ($\epsilon$--FNE) of the game~\eqref{eq: game1-cons} if
	\[
	\cX(\bar{\btheta}, \bar{\balpha}) \leq \epsilon^2 \quad \mbox{and} \quad \cY(\bar{\btheta}, \bar{\balpha}) \leq \epsilon^2,
	\]
	where 
	\begin{align*}
		\cX(\bar{\btheta}, \bar{\balpha}) \triangleq - 2 L_{11}\min_{ \btheta  \in \cSt}\,\,\Big[ & \langle \nabla_{\btheta} h(\bar{\btheta}, \bar{\balpha}), \btheta -\bar{\btheta} \rangle + q(\btheta) - q(\bar{\btheta}) 
	 +  \frac{L_{11}}{2} ||\btheta - \bar{\btheta}||^2  \Big],    
	\end{align*}\label{eq:X_k2}
	and
	\begin{align*}\label{eq:Y_k2}
	\cY(\bar{\btheta}, \bar{\balpha}) \triangleq 2 L_{22}\max_{\balpha \in \cSa}\,\,  \Big[ &	\langle \nabla_{\balpha} h(\bar{\btheta}, \bar{\balpha}), \balpha -\bar{\balpha} \rangle  - p(\balpha) + p(\bar{\balpha}) 
	-   \frac{L_{22}}{2} ||\balpha - \bar{\balpha}||^2 \Big].
	\end{align*}
\end{restatable}

\vspace{0.2cm}

In the unconstrained and smooth scenario that $\cSt = \mathbb{R}^{d_\theta}$, $\cSa =\mathbb{R}^{d_\alpha}$, and $p\equiv q \equiv 0$, the above $\epsilon$-FNE definition reduces to $\|{\nabla}_{\balpha} h(\bar{\btheta}, \bar{\balpha})\| \leq \epsilon$ and $\|{\nabla}_{\btheta} h(\bar{\btheta}, \bar{\balpha})\| \leq \epsilon$. 
\vspace{0.2cm}
\begin{rmk}\label{optimilaity_condtion}
The above  definition of $\epsilon$--FNE is stronger than the $\epsilon$-stationarity concept defined based on the proximal gradient norm in the literature (see, e.g.,~\cite{lin2019gradient}). Details of this remark is discussed in the Appendix section. 
\end{rmk}
\begin{rmk}
\label{remark:1}(Rephrased from Proposition 4.2 in \cite{pang2016unified}) For the min-max game~\eqref{eq: game1-cons}, under assumptions \ref{assumption:Concavity}, \ref{assumption:objective} and \ref{assumption: LipSmooth-uncons}, FNE always exists. Moreover, it is easy to show that $\cX(\cdot,\cdot)$ and $\cY(\cdot,\cdot)$ are continuous functions in their arguments. Hence, $\epsilon$--FNE exists for every $\epsilon \geq 0$.
\end{rmk}

\vspace{0.2cm}

In what follows, we consider two different scenarios for finding $\epsilon$-FNE points. In the first scenario, we assume that~$h(\btheta,\balpha)$ is strongly concave in $\balpha$ for every given $\btheta$ and develop a first-order algorithm for finding $\epsilon$-FNE. Then, in the second scenario, we extend our result to the case where $h(\btheta,\balpha)$ is concave (but not strongly concave) in $\balpha$ for every given $\btheta$.
\section{Non-Convex Strongly-Concave Games}
In this section, we study the zero-sum game \eqref{eq: game1-cons} in the case that the function $h(\btheta,\balpha)$ is $\sigma$-strongly concave in $\balpha$ for every given value of~$\btheta$. 
To understand the idea behind the algorithm, let us define the auxiliary function
\[
g(\btheta) \triangleq  \max\limits_{\balpha \in \cSa} h(\btheta, \balpha) - p(\balpha).
\]
A ``conceptual" algorithm for solving the min-max optimization problem~\eqref{eq: game1-cons} is to minimize the function $g(\btheta) + q(\btheta)$ using iterative decent procedures. First, notice that,  based on the following lemma, the strong concavity assumption implies the differentiability of~$g(\btheta)$. 

\begin{restatable}{lem}{lmggsmoot}\label{lm:gg-smooth}
	Let $g(\btheta) = \max\limits_{\balpha \in \cSa} h(\btheta, \balpha) - p(\balpha)$ in which the function $h(\btheta, \balpha)$ is $\sigma$-strongly concave in $\balpha$ for any given~$\btheta$. Then, under Assumption~\ref{assumption: LipSmooth-uncons}, the function $g(\btheta)$ is differentiable. Moreover, its gradient is $L_g$-Lipschitz continuous, i.e.,
	\[
\|\nabla g(\btheta_1) - \nabla g(\btheta_2)\| \leq L_g\| \btheta_1 - \btheta_2 \|,
\]
where $L_g = L_{11} + \dfrac{L_{12}^2}{\sigma}$.
\end{restatable}


The smoothness of the function~$g(\btheta)$ suggests the natural multi-step proximal method in Algorithm~\ref{alg: alg_grad} for solving the min-max optimization problem~\eqref{eq: game1-cons}. This algorithm performs two major steps in each iteration: the first major step, which is marked as ``Accelerated Proximal Gradient Ascent",  runs multiple iterations of the accelerated proximal gradient ascent to estimate the solution of the inner maximization problem. In other words, this step finds a point $\alpha_{t+1}$ such that
\vspace{-.2cm}
\[
\balpha_{t+1}\approx \arg\max_{\balpha \in \cSa} f(\btheta_t,\balpha).
\]
\normalsize
The output of this step will then be used to compute the approximate proximal gradient  of the function $g(\btheta)$ in the second step based on the classical Danskin's theorem~\cite{bernhard1995theorem, danskin1967theory}, which is restated below:

\vspace{0.2cm}

\begin{thm}
[Rephrased from \cite{bernhard1995theorem, danskin1967theory}]
\label{thm:danskin}
Let $V\subset \mathbb{R}^m$  be a compact set and $J(\mathbf{u},\bm{\nu}): \mathbb{R}^n \times V \mapsto \mathbb{R}$  be  differentiable with respect to $\textbf{u}$.  Let $\bar{J}(\textbf{u}) = \max\limits_{\bm{\nu} \in V}\; J(\textbf{u},\bm{\nu})$ and assume $\hat{V}(\textbf{u}) = \{\bm{\nu} \in V\;|\;J(\textbf{u},\bm{\nu}) = \bar{J}(\textbf{u})\}$ is singleton for any given $\mathbf{u}$. Then, $\bar{J}(\textbf{u})$ is differentiable and 
$\nabla_{\textbf{u}} \bar{J}(\textbf{u}) =  \nabla_{\textbf{u}} J(\textbf{u},\hat{\bm{\nu}})$ with $\hat{\bm{\nu}} \in \hat{V}(\mathbf{u})$. 
\end{thm}

\vspace{0.2cm}

According to the above lemma, the proximal gradient descent update rule on $g(\btheta)$ will be given by
\begin{align*} 
    \btheta_{t+1} = \arg\min\limits_{\btheta \in \cSt} \Big[& q(\btheta)+ \langle \nabla_{\btheta}h(\btheta_{t}, \balpha_{t+1}), \btheta - \btheta^{t}\rangle
     + \frac{L_g}{2} \|\btheta - \btheta_{t}\|^2\Big].
\end{align*}
The two main proximal  gradient update operators used in Algorithm~\ref{alg: alg_grad}  are defines as

\[
\rho_{\balpha}(\tilde{\btheta},\tilde{\balpha}, \gamma_1) = \arg\max\limits_{\balpha \in \cSa} \;\; \langle \nabla_{\balpha}
		h(\tilde{\btheta},\tilde{\balpha}), \balpha -\tilde{\balpha}\rangle  -\frac{\gamma_1}{2}\|\balpha -\tilde{\balpha}\|^2 -p(\balpha)
		\]
	and 
	\[\rho_{\btheta}(\tilde{\btheta},\tilde{\balpha}, \gamma_2) = \arg\min\limits_{\btheta \in \cSt} \;\; \langle \nabla_{\btheta}h(\tilde{\btheta}, \tilde{\balpha}), \btheta - \tilde{\btheta}\rangle  + \frac{\gamma_2}{2} \|\btheta - \tilde{\btheta}\|^2 + q(\btheta).
	\]
 \normalsize
		

\vspace{0.2cm}

The following theorem establishes the rate of convergence of  Algorithm~\ref{alg: alg_grad} to $\epsilon$-FNE. A more detailed statement of the theorem (which includes the constants of the theorem) is presented in the Appendix section.

\vspace{0.2cm}

\begin{thm}\label{thm:main}[Informal Statement]
	Consider the min-max zero-sum game
\begin{align*}
\min_{\btheta \in \cSt}\; \max_{\balpha \in \cSa} \;\;\bigg(f(\btheta,\balpha) = h(\btheta,\balpha) -p(\balpha) + q(\btheta)\bigg),
\end{align*}
where function $h(\btheta,\balpha)$ is $\sigma-$strongly concave in $\balpha$ for any given $\btheta$. In Algorithm \ref{alg: alg_grad}, if we choose $\eta_1 = \frac{1}{L_{22}}, \eta_2 = \frac{1}{L_g}$, $N = \sqrt{8L_{22}/\sigma}-1$; and $K$ and $T$ large enough such that	
	\[T \geq N_T(\epsilon) \triangleq {\cal O}(\epsilon^{-2}) \quad {\rm and} \quad	K \geq N_K(\epsilon) \triangleq {\cal O}(\log\big(\epsilon^{-1})\big),\]
	then there exists an iterate $t\in \{0,\cdots, T-1\}$ such that $(\btheta_t,\balpha_{t+1})$ is an $\epsilon$--FNE of \eqref{eq: game1-cons}. 
\end{thm}

\vspace{0.2cm}

\begin{center}

\begin{minipage}{0.85\textwidth}
\begin{algorithm}[H]
	\caption{Multi-step Accelerated Proximal Gradient Descent-Ascent} 
	\label{alg: alg_grad}
	\begin{algorithmic}[1]
		\State \textbf{Input}:  $K$, $T$, $N$, $\eta_1$, $\eta_2$, $\balpha_0 \in \cSa$ and $\btheta_0 \in \cSt$.
		

		\For  {$t=0, \cdots, T-1$} 
		\vspace{.1cm}
	    		\For {$k=0, \cdots, \lfloor{K/N}\rfloor$}
	    		 \tikzmark{top}
	    		 \State Set $\beta_1 = 1$ and $\bx_0 = \balpha_t$
        		     \If {$ k = 0$}
                        \State $\by_1 = \bx_0$
                        \Else
                        \State $\by_1 = \bx_N $
                    \EndIf

		            \vspace{.1cm}
        			\For{$j= 1, 2, \ldots, N$} 
		\State Set $\bx_{j} = \rho_{\balpha}(\btheta_{t},\by_j,\eta_1)$ 
		\State Set $\beta_{j+1} =\dfrac{1 + \sqrt{1+4\beta_j^2}}{2}$
		\State $\by_{j+1}  = \bx_j + \Big(\dfrac{\beta_j -1}{\beta_{j+1}} \Big)(\bx_j - \bx_{j-1})$ \tikzmark{right}
        	
    		\EndFor 
		
		\EndFor \tikzmark{bottom}
		\State  $\balpha_{t+1} = \bx_N$
		\State  $\btheta_{t+1} = \rho_{\btheta}(\btheta_{t},\balpha_{t+1},\eta_2)$\label{alg1:update-theta}
		\EndFor
	\end{algorithmic}
	\AddNote{top}{bottom}{right}{{\small  Accelerated 
	Proximal
	Gradient  Ascent
	\cite{nesterov1998introductory, beck2009fast}}}
\end{algorithm}
\end{minipage}
\end{center}

\vspace{0.2cm}

\begin{restatable}{cor}{Corollary}
Based on Theorem~\ref{alg: alg_grad},  to find an $\epsilon$-FNE of the game~\eqref{eq: game1-cons}, Algorithm~\ref{alg: alg_grad} requires ${\cal O}(\epsilon^{-2}\log(\epsilon^{-1}))$  gradients evaluations of the objective function. 
\end{restatable}

\section{Non-Convex Concave Games}
In this section, we consider the min-max problem~$\eqref{eq: game1-cons}$ under the assumption that $h(\btheta,\balpha)$ is concave (but not strongly concave) in $\balpha$ for any given value of~$\btheta$. In  this case, the direct extension of Algorithm~\ref{alg: alg_grad} will not work since the function $g(\btheta)$ might be non-differentiable. To overcome this issue, we start by making the function~$f(\btheta,\balpha)$ strongly concave by adding a ``negligible" regularization. More specifically, we define 
\begin{align}
    f_{\lambda}(\btheta,\balpha) = f(\btheta,\balpha) -\frac{\lambda}{2} \|\balpha - \hat{\balpha}\|^2, 
\end{align}
for some $\hat{\balpha} \in \cSa$.  
We then apply Algorithm~\ref{alg: alg_grad} to the modified non-convex-strongly-concave game
\begin{equation}\label{eq:game_reg}
\min_{\btheta\in\cSt}\max_{\balpha \in \cSa} \;f_{\lambda}(\btheta,\balpha).
\end{equation}
It can be shown that by choosing $\lambda = \frac{\epsilon}{2\sqrt{2}R}$, when we apply Algorithm~\ref{alg: alg_grad} to the modified game~\eqref{eq:game_reg}, we obtain an $\epsilon$-FNE of the original problem~\eqref{eq: game1-cons}. 
More specifically, with a proper choice of parameters, the following theorem establishes that the proposed method  converges to $\epsilon$-FNE point of the original problem.

\begin{restatable}{thm}{Thmmm}\label{thm: FW} [Informal Statement]
Set $\eta_1 = 1/(L_{22} + \lambda)$, $\eta_2 = 1/(L_{11} + L_{12}^2/\lambda)$, $\lambda = \dfrac{\epsilon}{2R}$, $N = \sqrt{8L_{22}/\lambda}-1$, and apply Algorithm~\ref{alg: alg_grad} to the regularized min-max problem~\eqref{eq:game_reg}. Choose $K,T$ large enough such that
$
T \geq N_T(\epsilon) \triangleq {\cal O} (\epsilon^{-3}), \;  {\rm and} \;	K \geq N_K(\epsilon) \triangleq {\cal O}\big(\epsilon^{-1/2} \log(\epsilon^{-1})\big).
$
Then, there exists $t \in \{0, \ldots, T-1\}$ in Algorithm~\ref{alg: alg_grad} such that $(\btheta_t, \balpha_{t+1})$ is an $\epsilon$-FNE of the original problem~\eqref{eq: game1-cons}. 
\end{restatable}
\begin{restatable}{cor}{Corollaryy}
Based on Theorem~\ref{thm: FW},  Algorithm~\ref{alg: alg_grad} requires ${\cal O}(\epsilon^{-3.5}\log(\epsilon^{-1}))$ gradient evaluations in order to find a $\epsilon$-FNE of the game~\eqref{eq: game1-cons}. 
\end{restatable}

 
 \vspace{-.cm}
 
\section{Numerical Experiments}
In this section, we evaluate the performance of the proposed algorithm for the problem of attacking the LASSO estimator. In other words, our goal is to find a small perturbation of the observation matrix that  worsens the performance of the LASSO estimator in the training set. 
This attack problem can be formulated as
\begin{align}\label{eq:experiment}
    \max\limits_{\textbf{A} \in \mathcal{B}(\hat{\textbf{A}}, \Delta)} \min\limits_{\bx} \|\textbf{A}\bx - \textbf{b}\|_2^2 + \xi \|\bx\|_1,
\end{align}

where $\mathcal{B}(\hat{\textbf{A}}, \Delta) = \{\textbf{A}\;|\;||\textbf{A} - \hat{\textbf{A}}||^2_{F} \leq \Delta\}$ and the matrix $\textbf{A} \in \mathbb{R}^{m \times n}$. We set $m = 100$, $n = 500$, $\xi = 1$ and $\Delta = 10^{-1}$. In our experiments, first we generate a ``ground-truth" vector~$\bx^*$ with sparsity level $s = 25$ in which the location of the non-zero elements are chosen randomly and their values are sampled from a standard Gaussian distribution.
Then, we generate the elements of matrix $\mathbf{A}$ using standard Gaussian distribution. Finally, we set $\mathbf{b} = \mathbf{A} \mathbf{x}^* + \mathbf{e}$, where $\mathbf{e} \sim N(\mathbf{0}, 0.001\mathbf{I})$. We compare the performance of the proposed algorithm with the popular  subgradient descent-ascent and  proximal gradient descent-ascent algorithms. In the subgradient descent-ascent algorithm, at each iteration, we take one step of sub-gradient ascent step with respect to~$\bx$ followed by one steps of sub-gradient ascent in $\mathbf{A}$. Similarly,  each iteration of the proximal gradient descent-ascent algorithm consists of one step of proximal gradient descent with respect to~$\bx$ and one step of proximal gradient descent with respect to~$\mathbf{A}$.

\vspace{0.2cm}

To have a fair comparison, all of the studied algorithms have been  initialized at the same random points in  Fig.~\ref{fig:plot}. 

\begin{figure}[H]
\centering
\includegraphics[width=.5\textwidth]{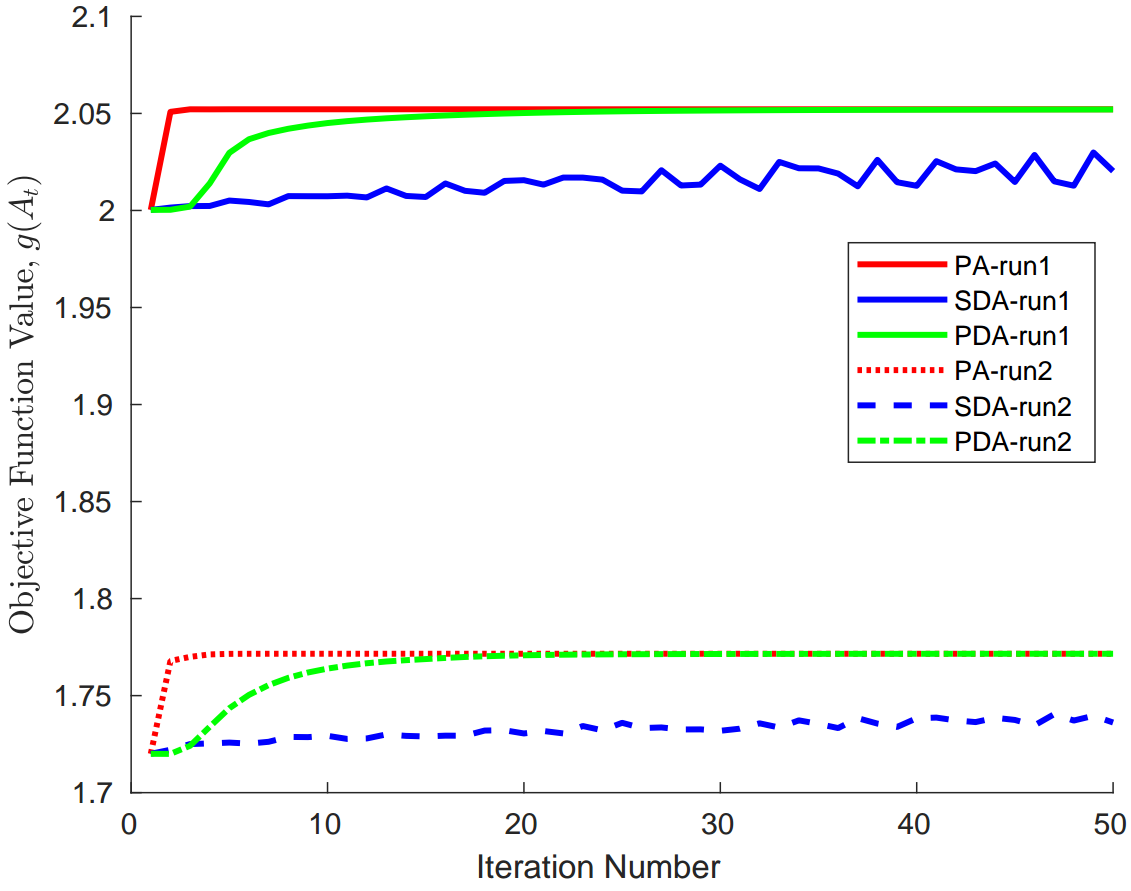}\includegraphics[width=.5\textwidth]{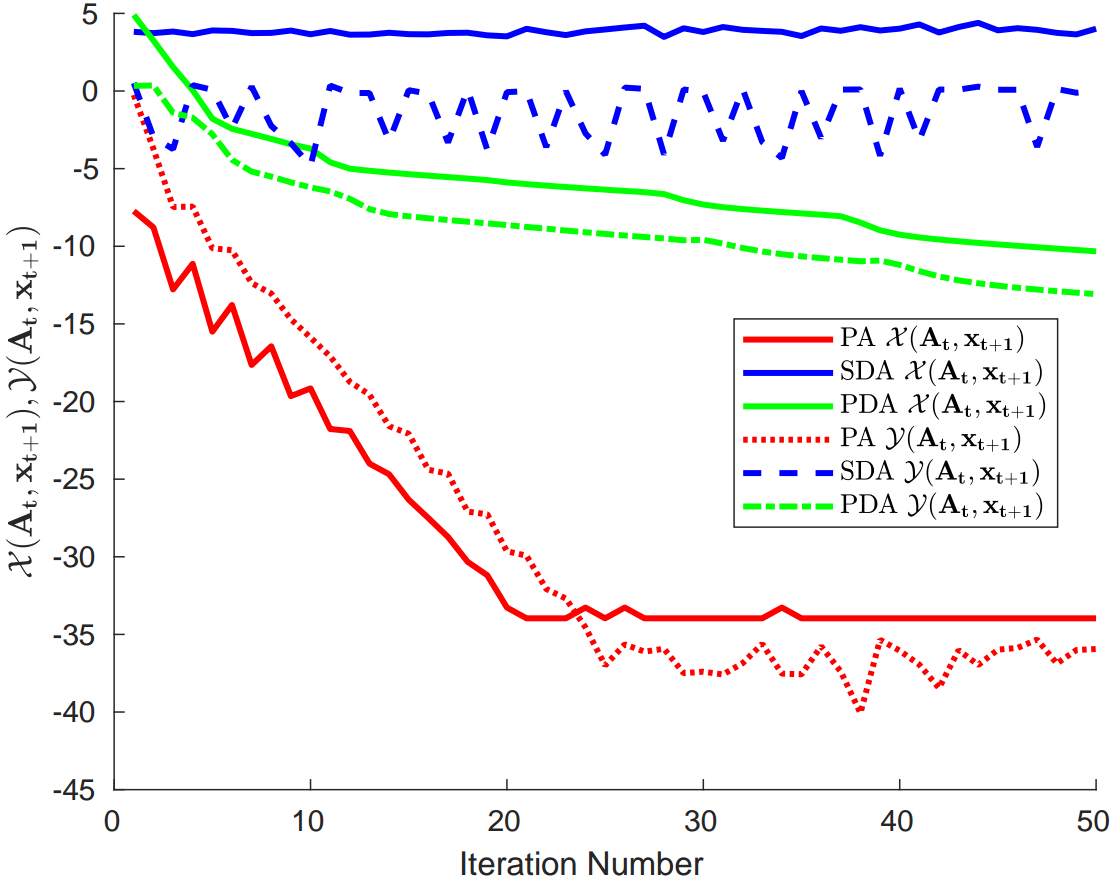}\quad
\caption{ \small (left): Convergence behavior of different algorithms in terms of the objective value. The objective value at iteration $t$ is defined as $g(\mathbf{A}_t)\triangleq \min_{\bx} \|\mathbf{A}_t \bx  - \mathbf{b}\|_2^2 + \xi \|\bx\|_1$, (right): Convergence behavior of different algorithms in terms of the stationarity measures~$\cX(\mathbf{A}_{t},\bx_{t+1})$, $\cY(\mathbf{A}_{t},\bx_{t+1})$ (logarithmic scale). The list of the algorithms used in the comparison is as follows: Proposed Algorithm (PA), Subgradient Descent-Ascent (SDA), and Proximal  Descent-Ascent algorithm (PDA).}
\label{fig:plot}
\end{figure}




The above figure might not be a fair comparison since each step of the proposed algorithm is computationally more expensive than the two benchmark methods. To have a better comparison, we evaluate the performance of the algorithms in terms of the required time for convergence. Table~\ref{table:1} summarizes the average time required for different algorithms for finding a point $(\bar{\mathbf{A}},\bar{\bx})$ satisfying $\cX(\bar{\mathbf{A}},\bar{\bx})\leq 0.1$ and $\cY(\bar{\mathbf{A}},\bar{\bx})\leq 0.1$. The average is taken over 100 different experiments. As can be seen in the table, the proposed method in average converges an order of magnitude faster than the other two algorithms. 

\begin{table}[H]
\begin{center}
\begin{tabular}{ |c|c|c|c| } 
 \hline
 Algorithm & PA & SDA & PDA \\ 
 \hline
 Average time (seconds) & 0.0268 & 3.5016 & 0.5603 \\ 
 \hline
 Standard deviation (seconds) & 0.0538 & 7.0137 & 1.1339 \\
 \hline
\end{tabular}
\caption{\footnotesize Average computational time of different algorithms. 
}
\label{table:1}
\end{center}
\end{table}

\vspace{-0.5cm}

\section*{Acknowledgement}
The authors would like to thank Shaddin Dughmi and Dmitrii M. Ostrovskii for their insightful comments that helped to improve the work.   



\bibliographystyle{IEEEbib}

\bibliography{strings}
 

\normalsize
\section*{Appendix}

\normalsize

\noindent\textbf{Discussions on Remark~\ref{optimilaity_condtion}:}
Consider the optimization problem 
\begin{align}\label{opt}
 \min\limits_{\bz\in \cZ}\;\; F(\bz),   
\end{align}
in which the set $\cZ$ is bounded and convex; and $F(\cdot): \mathbb{R}^n \mapsto \mathbb{R}$ is $\ell$-smooth, i.e.,
\[
\|\nabla F(\bz_1) - \nabla F(\bz_2)\| \leq \ell \| \bz_1- \bz_2\|.
\]
One of the commonly used definitions of  $\epsilon$-stationary point  for the optimization problem~\eqref{opt} is as follows.  
\begin{defi}
[$\epsilon$-stationary point of the first type] A point $\bar{\bz}$ is said to be an $\epsilon$-stationary point of the first type of~\eqref{opt} if
\begin{align}
& \bigg\|\mathcal{P}_{\cZ} \left(\bar{\bz} - \frac{1}{\ell}\nabla F(\bar{\bz})\right) - \bar{\bz}\bigg\| \leq \frac{\epsilon}{\ell} \label{eq:SP1T},
\end{align}
where $\mathcal{P}_{\cZ}(\cdot)$ represents the projection operator to the feasible set $\cZ$.
\end{defi}
\vspace{0.2cm}

 Another  notion of stationarity, which is used in this paper (as well as other works including~\cite{karimi2016linear}), is defined as follows. 

 \vspace{0.2cm}

 \begin{defi}
[$\epsilon$-stationary point of the second type] A point $\bar{\bz}$ is said to be an $\epsilon$-stationary point of the second type for the optimization problem~\eqref{opt} if
\begin{align}
D(\bar{\bz}) \leq \epsilon^2, \label{eq:SP2T}
\end{align}
where $D(\bar{\bz}) \triangleq -2\ell\min\limits_{\bz \in \cZ} \left[\langle \nabla F(\bar{\bz}), \bz-\bar{\bz}\rangle + \frac{\ell}{2} \|\bz-\bar{\bz}\|^2\right]$.
\end{defi}
 
 The following theorem shows that the stationarity definition in~\eqref{eq:SP2T} is strictly stronger than the stationarity definition in~\eqref{eq:SP1T}.


\begin{thm}
The $\epsilon$-stationary concept of the second type is stronger than the $\epsilon$-stationary concept of the first type. In particular, if a point $\bar{\bz}$ satisfies~\eqref{eq:SP2T}, then it must also satisfy \eqref{eq:SP1T}. Moreover, there exist an optimization problem with a given feasible point $\bar{\bz}$ such that $\bar{\bz}$ is $\epsilon$-stationary point of the first type, but it is not $\epsilon'$-stationary point of the second type for any $\epsilon' < \sqrt{2\epsilon + \epsilon^2}$.
\end{thm}
\begin{proof}
We first show that \eqref{eq:SP2T} implies \eqref{eq:SP1T}, i.e., if $D(\bar{\bz}) \leq \epsilon^2$ then $ \|\mathcal{P}_{\cZ} \left(\bar{\bz} + (1/\ell)\nabla F( \bar{\bz})\right)\| \leq \epsilon/\ell$. 
From definition of $D(\bar{\bz})$, we have 
\begin{align*}
   	D(\bar{\bz}) \triangleq & -2 \ell\min_{\bz \in \cZ}\,\,  \Big[ 	\langle \nabla  F(\bar{\bz}), \bz -\bar{\bz} \rangle+\frac{\ell}{2} ||\bz - \bar{\bz}||^2 \Big]
   	\\ = &  -\ell^2\min_{\bz \in \cZ}\,\,  \Big[ \frac{2}{\ell}\langle \nabla  F( \bar{\bz}), \bz -\bar{\bz} \rangle +||\bz - \bar{\bz}||^2 \Big]
   	\\ = &  -\ell^2\min_{\bz \in \cZ}\,\,  \Big[ \|\bz -\bar{\bz} + \frac{1}{\ell} \nabla  F(\bar{\bz}) \|^2 - \frac{1}{\ell^2} \|\nabla  F(\bar{\bz})\|^2  \Big]
   	   	\\ = & - \ell^2\; \|\mathcal{P}_{\cZ}(\bar{\bz} - \frac{1}{\ell } \nabla  F(\bar{\bz}))- (\bar{\bz} - \frac{1}{\ell} \nabla  F(\bar{\bz})) \|^2+ \|\nabla  F(\bar{\bz})\|^2.
\end{align*}

Defining $\hat{\bz} = \bar{\bz} - \frac{1}{\ell } \nabla  F(\bar{\bz})$, we get
\begin{align}
D(\bar{\bz}) =  - \ell^2\; \|\mathcal{P}_{\cZ}(\hat{\bz})- \hat{\bz} \|^2+ \|\nabla  F(\bar{\bz})\|^2.  \label{eq:Dalpha}
\end{align}

On the other hand, as shown in Fig.~\ref{fig:general},  the direct application of cosine equality implies that
\begin{align}\label{eq:cosine}
\|\frac{1}{\ell}\nabla  F(\bar{\bz})\|^2 = \|\mathcal{P}_{\cZ}(\hat{\bz}) - \bar{\bz}\|^2 + \|\mathcal{P}_{\cZ}(\hat{\bz}) - \hat{\bz}\|^2 - 2 (\|\mathcal{P}_{\cZ}(\hat{\bz})- \bar{\bz}\|) (\|\mathcal{P}_{\cZ}(\hat{\bz}) - \hat{\bz}\|)\cos{\gamma},
\end{align}
where $\gamma$ is the angle between the two vectors~$ \bar{\bz} - \mathcal{P}_{\cZ}(\hat{\bz})$ and $\hat{\bz} - \mathcal{P}_{\cZ}(\hat{\bz})$. Moreover, from \cite[Lemma 3.1]{bubeck2015convex} we know that $\cos{\gamma} \leq 0$. As a result,
\begin{align*}
\ell^2\|\mathcal{P}_{\cZ}(\hat{\bz})- \bar{\bz}\|^2 \leq - \ell^2\; \|\mathcal{P}_{\cZ}(\hat{\bz})- \hat{\bz} \|^2+ \|\nabla  F(\bar{\bz})\|^2 = D(\bar{\bz}),
\end{align*}
where the last equality is due to~\eqref{eq:Dalpha}.
Furthermore, since  $\bar{\bz}$ is an $\epsilon$-stationery point, i.e., $D(\bar{\bz}) \leq \epsilon^2$, we conclude that $\|\mathcal{P}_{\cZ}(\hat{\bz})- \bar{\bz}\| \leq \epsilon/\ell$. In other words, $\bar{\bz}$ is an $\epsilon$-stationary point of the first type. 
\begin{figure}[H]
\begin{minipage}[b]{1.0\linewidth}
  \centering
  \centerline{\includegraphics[width=7.5cm]{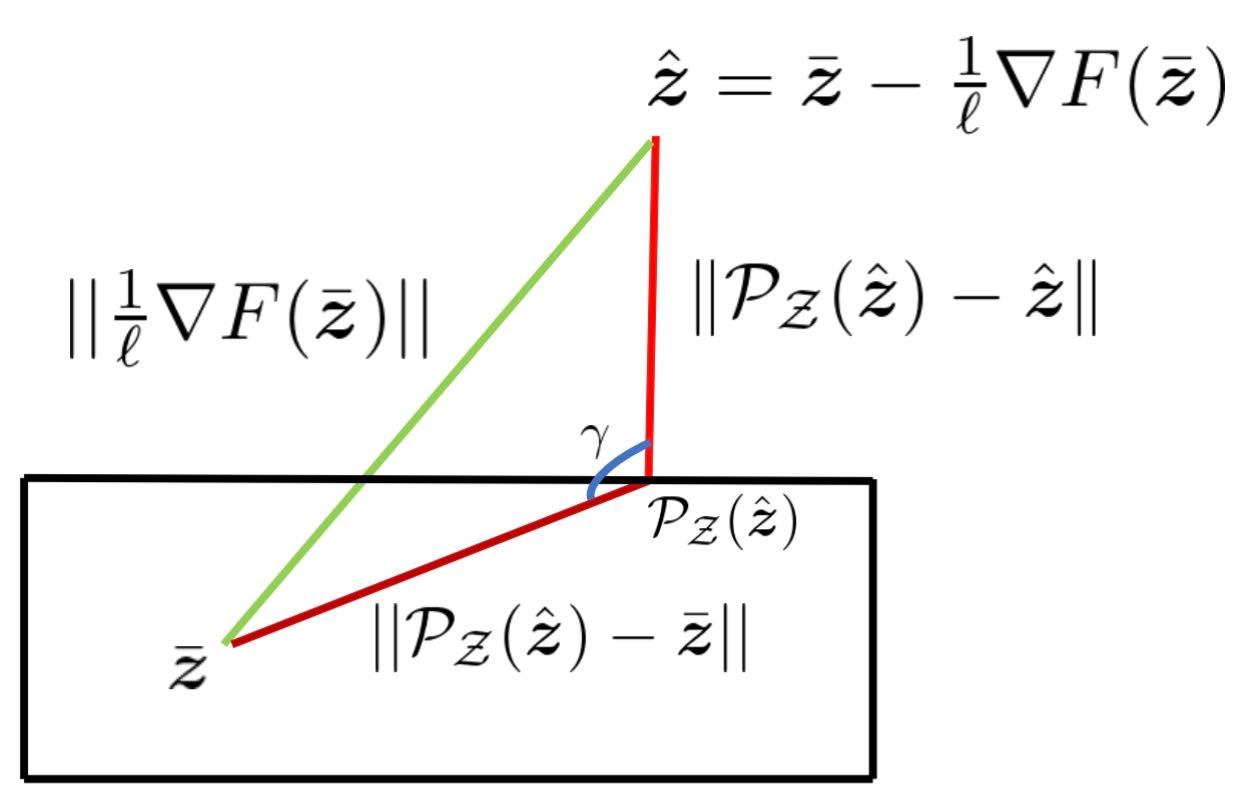}}
\end{minipage}
\caption{Relation between different notions of stationarity}
\label{fig:general}
\end{figure}


Next we show that the stationarity concept in~\eqref{eq:SP2T} is strictly stronger than the stationarity concept in~\eqref{eq:SP1T}. To understand this, let us take an additional look at Fig.~\ref{fig:general} and equation~\eqref{eq:cosine} used in the proof above. Clearly, the two stationarity measures could coincide when $\cos{\gamma} = 0$. Moreover, the two notions have the largest gap when~$\cos{\gamma} = -1$. Fig.~\ref{fig:extreme_cases} shows both of these scenarios. 
\normalsize

\begin{figure}[H]
\centering
\includegraphics[width=.3\textwidth]{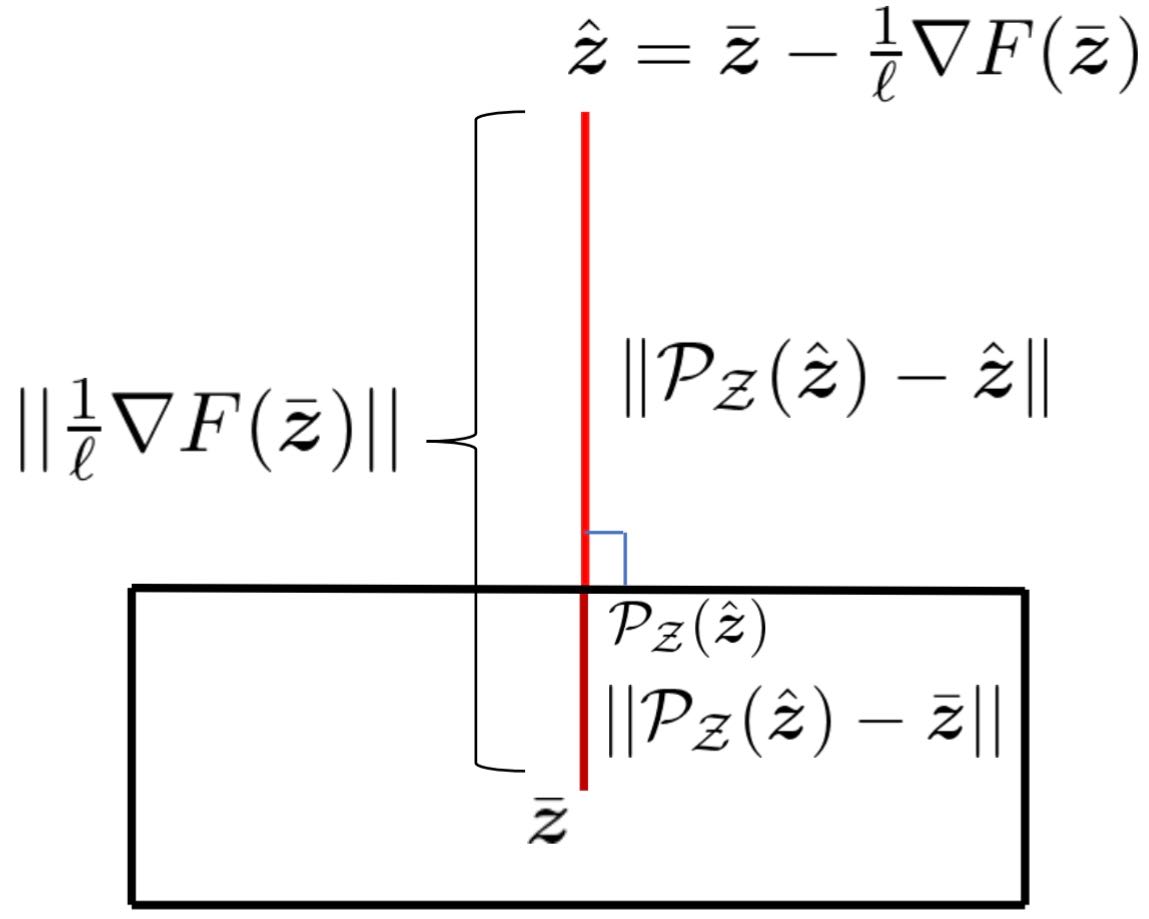}\qquad\qquad\qquad\quad\includegraphics[width=.3\textwidth]{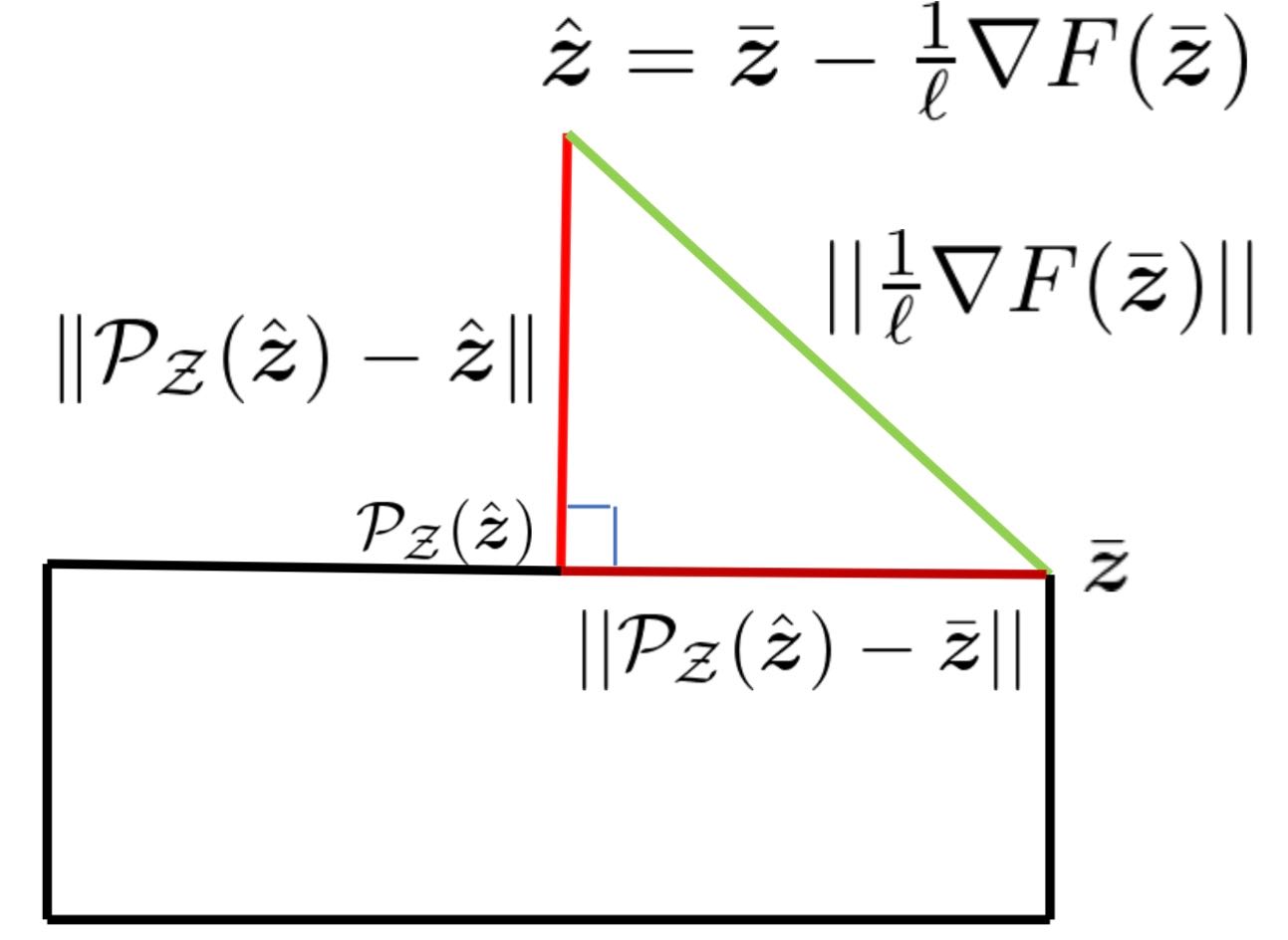}\quad
\caption{ \small (left): $\gamma = \pi$, two measures have the largest deviation, (right): $\gamma =\frac{\pi}{2}$, two measures coincide}
\label{fig:extreme_cases}
\end{figure}
\normalsize
According to Fig.~\ref{fig:extreme_cases}, in order to create an example with largest gap between the two stationarity notions, we need to construct an example with the smallest possible value of $\cos{\gamma}$. In particular, consider the optimization problem
\[
\min_z \;\; \frac{1}{2} z^2 \quad\quad \textrm{s.t.} \;\; z\geq 1.
\]
It is easy to check that the point $\bar{z} = 1+\epsilon$ is an $\epsilon$-stationary point of the first type, while it is not an $\epsilon'$-stationary point of the second type for any $\epsilon' < \sqrt{2\epsilon + \epsilon^2}$.
\end{proof}
\vspace{0.2cm}

Next, we re-state the lemmas used in the main body of the paper and present detailed proof of them.

\vspace{0.2cm}
\lmggsmoot*
\normalsize

\begin{proof}

The differentiability of the function $g(\btheta)$ is obvious from Danskin's Theorem~\ref{thm:danskin}. In order to find the gradient's Lipschitz constant, define $l(\btheta,\balpha) = -h(\btheta, \balpha) + p(\balpha)$. Let 
\[\balpha_1^* = \argmin_{\balpha \in \cSa} \;l (\btheta_1, \balpha)\quad {\rm  and } \quad \balpha_2^* = \argmin_{\balpha \in \cSa} \;l (\btheta_2, \balpha).\]
Due to $\sigma$--strong convexity of $l(\btheta,\balpha)$ in $\balpha$ for any given $\btheta$, we have 

\[\begin{array}{ll}
l(\btheta_2, \balpha_2^*)& \geq l(\btheta_2, \balpha_1^*) 
+  l^{'}(\btheta_2, \balpha_1^*; \balpha^*_2 - \balpha^*_1) + \dfrac{\sigma}{2}\|\balpha_2^* - \balpha_1^*\|^2,\end{array}\]
and
\[\begin{array}{ll}
l(\btheta_2, \balpha_1^*)& \geq  l(\btheta_2, \balpha_2^*) + l^{'}(\btheta_2, \balpha_2^*; \balpha^*_1 - \balpha^*_2) + \dfrac{\sigma}{2}\|\balpha_2^* - \balpha_1^*\|^2.\end{array}\]

Furthermore, due to optimality of  $\balpha_2^*$, $l^{'}(\btheta_2, \balpha_2^*; \balpha^*_1 - \balpha^*_2) \geq 0$. As a result, by adding the above two inequalities we have
\begin{equation}\label{eq: Lipschtizness-1}
-l^{'}(\btheta_2, \balpha_1^*; \balpha^*_2 - \balpha^*_1) \geq \sigma\|\balpha_2^* - \balpha_1^*\|^2.
\end{equation}

On the other hand, from optimality of $\balpha_1^*$, we have
\begin{equation}\label{eq: Lipschtizness-2}
l^{'}(\btheta_1, \balpha_1^*; \balpha^*_2 - \balpha^*_1)\geq 0.
\end{equation}

Now, by adding ~\eqref{eq: Lipschtizness-1} and~\eqref{eq: Lipschtizness-2} we get
\begin{equation}
\begin{array}{ll}
&\sigma\|\balpha_2^* - \balpha_1^*\|^2 \\ & \leq 
-l^{'}(\btheta_2, \balpha_1^*; \balpha^*_2 - \balpha^*_1) + l^{'}(\btheta_1, \balpha_1^*; \balpha^*_2 - \balpha^*_1)\\
& = \langle  \nabla_{\balpha} h(\btheta_2, \balpha^*_1), \balpha^*_2 - \balpha^*_1 \rangle - p^{'}(\balpha^*_1;\balpha^*_2 - \balpha^*_1 ) - \langle  \nabla_{\balpha} h(\btheta_1, \balpha^*_1), \balpha^*_2 - \balpha^*_1 \rangle + p^{'}(\balpha^*_1;\balpha^*_2 - \balpha^*_1 )\\
& =  \langle  \nabla_{\balpha} h(\btheta_2, \balpha^*_1) - \nabla_{\balpha} h(\btheta_1, \balpha^*_1), \balpha^*_2 - \balpha^*_1 \rangle \\\nonumber
&  \leq L_{12} \|\btheta_1 - \btheta_2\|\|\balpha_2^* - \balpha_1^*\|, \; \;
\end{array}
\end{equation}
where the last inequality holds by Cauchy-Schwartz and the Lipschitzness from Assumption~\ref{assumption: LipSmooth-uncons}.
As a result, we get

\begin{align}\label{eq: Lipschtizness-3}
 \|\balpha_2^* - \balpha_1^*\| \leq \frac{L_{12}}{\sigma} \|\btheta_1 - \btheta_2\|.
\end{align}

Now, Theorem~\ref{thm:danskin} implies that  
\[\arraycolsep=1pt\def\arraystretch{1.6}
\begin{array}{ll}
\|\nabla_{\btheta} g(\btheta_1) - \nabla_{\btheta} g(\btheta_2)\| &= \|\nabla_{\btheta}  h(\btheta_1, \balpha_1^*) - \nabla_{\btheta} h(\btheta_2, \balpha_2^*) \|\\
&= \|\nabla_{\btheta} h(\btheta_1, \balpha_1^*) -\nabla_{\btheta} h(\btheta_2, \balpha_1^*) +\nabla_{\btheta} h(\btheta_2, \balpha_1^*) - \nabla_{\btheta} h(\btheta_2, \balpha_2^*) \|\\
&\leq L_{11}\|\btheta_1 - \btheta_2\| + L_{12}\|\balpha_1^* - \balpha_2^*\|\\
&\leq \Big(L_{11} + \dfrac{L_{12}^2}{\sigma}\Big)\|\btheta_1 - \btheta_2\|,	
\end{array}\]
where the last inequality is due to~\eqref{eq: Lipschtizness-3}.	
\end{proof}

	

 
\begin{restatable}{lem}{remarking}\label{lm:projected-GD} (Rephrased from~\cite{nesterov2013gradient,beck2009fast})
	Assume $F(\bx) = m(\bx) + n(\bx)$, where $m(\bx)$ is $\sigma$-strongly convex and $L$-smooth, $n(\bx)$ is convex and possibly non-smooth (and possibly extended real-valued). 
	Then, by applying accelerated proximal gradient descent algorithm with restart parameter $N\triangleq \sqrt{8L/\sigma} -1$ for $K$ iterations, with $K$ being a constant multiple of $N$, we get 
	\begin{equation}\label{eq:projected-GD}
	F(\bx_K) - F(\bx^*) \leq \left(\dfrac{1}{2}\right)^{K/N}(F(\bx_0) - F(\bx^*)),
	\end{equation}
	where $\bx_K$ is the iterate obtained at iteration $K$ and $\bx^* \triangleq \argmin\limits_{\bx}\; F(\bx)$.
\end{restatable}

\begin{restatable}{lem}{errorbound}\label{lm:error_bound}
 Let $\balpha_{t+1}$ to be the output of the accelerated proximal gradient descent in Algorithm~\ref{alg: alg_grad} at iteration $t$.  Assume  $\kappa = \dfrac{L_{22}}{\sigma} \geq 1$, and  $\gl(\btheta_t) - \left(h(\btheta_{t}, \balpha_{0}(\btheta_{t})) - p(\balpha_{0}(\btheta_{t}))\right) < \Delta$. Then for any prescribed $\epsilon \in (0,1)$, choose $K$ large enough such that 
\begin{equation}\label{K-value}\nonumber
K \geq  2 \sqrt{8\kappa}\left(\log L_{22} + \log \;( 2L_{22} R + g_{\max} + L_p + R) + 2 \log \;\left(\frac{1}{\epsilon}\right) + \frac{1}{2} \log \;\left(\frac{2\Delta}{\sigma}\right)\right),
\end{equation}

	where $g_{max} = \max\limits_{\balpha \in \cSa }\|\nabla_{\balpha} h(\btheta_t,\balpha)\|$. Then the error $e_t \triangleq \hl- \nabla \gl(\btheta_t)$ has a norm 
	\begin{equation*}
	\|e_t\|\leq \delta \triangleq \frac{L_{12}}{L_{22}(2L_{22}R + g_{\max} + L_p + R)}\epsilon^2
	\end{equation*}
	and
	\begin{equation*}
	 \epsilon^2 \geq \cY(\bt)  \triangleq L_{22} \max_{\balpha \in \cSa}\,\, 	\langle \hla, \balpha -\balpha_{t+1} \rangle  - p(\balpha) + p(\balpha_{t+1}) - \frac{L_{22}}{2} ||\balpha - \balpha_{t+1}||^2 
	\end{equation*}
\end{restatable}

\begin{proof}
    From Lemma~\ref{lm:projected-GD} we have, 
	\begin{equation}\label{eq:Updating-alpha-1}
	\gl(\btheta_t) - \big(h(\bt) - p(\balpha_{t+1}) \big)\leq \dfrac{1}{2^{\frac{K}{\sqrt{8\kappa}}}}\Delta.
	\end{equation}

	Let $\balpha^*(\btheta_t) \triangleq \arg\max\limits_{\balpha \in \cSa} \; h(\btheta_t,\balpha ) - p(\balpha) $. By combining~\eqref{eq:Updating-alpha-1} and strong concavity of $h(\btheta_t,\balpha ) - p(\balpha)$ in $\balpha$, we get
	
	\begin{equation*}\label{eq:Updating-alpha-2}
	\dfrac{\sigma}{2}\|\balpha_{t+1} - \balpha^*(\btheta_t)\|^2 \leq   \gl(\btheta_t) - \big(h(\bt) - p(\balpha_{t+1}) \big) \leq \dfrac{1}{2^{\frac{K}{\sqrt{8\kappa}}}}\Delta.
	\end{equation*}

Combining this inequality with  Assumption~\ref{assumption: LipSmooth-uncons} implies that 
	\begin{align*}\label{eq:Updating-alpha-3}
		\|e_t\|&=\| \hl- \nabla \gl(\btheta_t)\| =\| \hl - \hls \| \\
	&\leq L_{12} \|\balpha_{t+1} - \balpha^*(\btheta_t)\|\leq  \dfrac{L_{12}}{2^{K/2\sqrt{8\kappa}}} \sqrt{\dfrac{2\Delta}{\sigma}}\leq  \frac{L_{12}}{L_{22}(2L_{22}R + g_{\max} + L_p + R)}\epsilon^2,
	\end{align*}
	
where the last inequality comes from our choice of $K$.

\vspace{0.2cm}

Next, let us prove the second part of the lemma. First notice that by some algebraic manipulations, we can write
\begin{align*}
&\frac{1}{2L_{22}}\cY(\bt) \\
 = &  \max_{\balpha \in \cSa}\,\, \Big[	\langle \hla, \balpha -\balpha_{t+1} \rangle  - p(\balpha) + p(\balpha_{t+1}) - \frac{L_{22}}{2} ||\balpha -\balpha_{t+1} ||^2 \Big]
\\ = & \max_{\balpha \in \cSa}\,\, \Big[	\langle \hla, \balpha -\balpha_{t+1}  \rangle  - p(\balpha) + p(\balpha_{t+1}) - \frac{L_{22}}{2} ||\balpha -\balpha^*(\btheta_{t}) + \balpha^*(\btheta_{t}) - \balpha_{t+1}||^2 \Big]
\\ = &\max_{\balpha \in \cSa}\,\, \Big[	\langle \hla, \balpha -\balpha_{t+1}  \rangle  - p(\balpha) + p(\balpha_{t+1}) 
\\& \quad - \frac{L_{22}}{2} \|\balpha - \balpha^*(\btheta_t)\|^2  - \frac{L_{22}}{2} \|\balpha^*(\btheta_{t}) - \balpha_{t+1}\|^2 - L_{22}\langle \balpha - \balpha^*(\btheta_{t}), \balpha^*(\btheta_{t}) -  \balpha_{t+1} \rangle \Big]
\\= & \max_{\balpha \in \cSa}\,\, \Big[	\langle \hla, \balpha -\balpha_{t+1} \rangle  - p(\balpha) + p(\balpha_{t+1}) - \frac{L_{22}}{2} \|\balpha - \balpha^*(\btheta_t)\|^2  
\\ & \quad  - \frac{L_{22}}{2} \|\balpha^*(\btheta_{t}) - \balpha_{t+1}\|^2 - L_{22} \langle \balpha - \balpha^*(\btheta_{t}), \balpha^*(\btheta_{t}) -  \balpha_{t+1} \rangle - p(\balpha^*(\btheta_{t})) + p(\balpha ^*(\btheta_{t}))\Big].
\end{align*}

Thus, we obtain

\begin{align*}
&\frac{1}{2L_{22}}\cY(\bt) 
\\ \leq & \max_{\balpha \in \cSa}\,\, \Big[	\langle \hla, \balpha -\balpha_{t+1} \rangle  - p(\balpha) + p(\balpha_{t+1}) 
\\ & \quad - \frac{L_{22}}{2} \|\balpha - \balpha^*(\btheta_t)\|^2 - L_{22} \langle \balpha - \balpha^*(\btheta_{t}), \balpha^*(\btheta_{t}) -  \balpha_{t+1} \rangle - p(\balpha^*(\btheta_{t})) + p(\balpha ^*(\btheta_{t}))\Big]
\\  = &  \max_{\balpha \in \cSa}\,\, \Big[	\langle \hla - \hlss, \balpha -\balpha_{t+1} \rangle 
\\  & \quad + \langle \hlss, \balpha - \balpha_{t+1} \rangle - p(\balpha) + p(\balpha_{t+1})
\\  & \quad - \frac{L_{22}}{2} \|\balpha - \balpha^*(\btheta_t)\|^2 - L_{22}\langle \balpha - \balpha^*(\btheta_{t}), \balpha^*(\btheta_{t}) -  \balpha_{t+1} \rangle   - p(\balpha^*(\btheta_{t})) + p(\balpha ^*(\btheta_{t})) \Big] 
\\ \leq &  \max_{\balpha \in \cSa}\,\, 	\Big[\langle \hla - \hlss, \balpha -\balpha_{t+1} \rangle + \langle \hlss, \balpha^*(\btheta_{t}) - \balpha_{t+1} \rangle 
\\ & \quad + p(\balpha_{t+1})  - p(\balpha^*(\btheta_{t})) - L_{22}\langle \balpha - \balpha^*(\btheta_{t}), \balpha^*(\btheta_{t}) -  \balpha_{t+1} \rangle \Big] 
\\ & \quad + \underbrace{\max_{\balpha \in \cSa} \Big[ \langle \hlss, \balpha - \balpha^*(\btheta_{t}) \rangle - p(\balpha) + p(\balpha^*(\btheta_{t})) - \frac{L_{22}}{2}\|\balpha- \balpha^*(\btheta_{t})\|^2}_{= 0 } \Big]
\\    \leq &  \max_{\balpha \in \cSa}\,\, \Big[	\langle \hla - \hlss, \balpha -\balpha_{t+1} \rangle + \langle \hlss, \balpha^*(\btheta_{t}) - \balpha_{t+1} \rangle 
\\ &  \quad  + p(\balpha_{t+1})  - p(\balpha^*(\btheta_{t})) - L_{22}\langle \balpha - \balpha^*(\btheta_{t}), \balpha^*(\btheta_{t}) -  \balpha_{t+1} \rangle \Big]
\\ \leq &    L_{22} \|\balpha_{t+1} - \balpha^*(\btheta_{t})\| R  + g_{\max} \|\balpha_{t+1} - \balpha^*(\btheta_{t})\| + L_{p} \|\balpha_{t+1} - \balpha^*(\btheta_{t})\| + R L_{22}  \|\balpha_{t+1} - \balpha^*(\btheta_{t})\| 
\\  \leq & ( L_{22} R + g_{\max} + L_p + RL_{22}) \|\balpha_{t+1} - \balpha^*(\btheta_{t})\|.
\end{align*}

As a result,
\begin{align*}
\cY(\bt) & \leq L_{22}(2L_{22} R + g_{\max} + L_p ) \|\balpha_{t+1} - \balpha^*(\btheta_{t})\| 
\\ & \leq L_{22}(2L_{22} R + g_{\max} + L_p )\frac{1}{{2^{K/2\sqrt{8\kappa}}}}  \sqrt{\dfrac{2\Delta}{\sigma}}
\\ & \leq \epsilon^2,
\end{align*}
where the last inequality follows from the choice of $K$.
\end{proof}

\noindent\textbf{Theorem~\ref{thm:main}.} [Formal Statement]
{\it	Consider the min-max zero sum game
\begin{align*}
\min_{\btheta \in \cSt}\; \max_{\balpha \in \cSa} \;\;(f(\btheta,\balpha) = h(\btheta,\balpha) -p(\balpha) + q(\btheta)),
\end{align*}
where the function $h(\btheta,\balpha)$ is $\sigma-$strongly concave. Let $D = g(\btheta_0) + q(\btheta_0) - \displaystyle{\min_{\btheta\in \Theta}}\left(g(\btheta) + q(\btheta)\right)$ where $g(\btheta) = \max_{\balpha\in\mathcal{A}} h(\btheta,\balpha) - p(\balpha)$, and $L_g = L_{11} + \frac{L_{12}^2}{\sigma}$ be the Lipschitz constant of the gradient of $g$. In  Algorithm~\ref{alg: alg_grad}, if we set $\eta_1 = \frac{1}{L_{22}}, \eta_2 = \frac{1}{L_g}$, $N = \sqrt{8L_{22}/\sigma}-1$ and choose  $K$ and $T$ large enough such that	
	\[
	T \geq N_T(\epsilon) \triangleq  \frac{4L_g D}{\epsilon^2}
	\]
 and
\[
 K \geq N_K(\epsilon) \triangleq  2 \sqrt{8\kappa}\Bigg(C + 2 \log \;\left(\frac{1}{\epsilon}\right) + \frac{1}{2} \log \;\left(\frac{2\Delta}{\sigma}\right)\Bigg),
  \]
  where $C = \max \Big\{2\log 2 + \log \;( L_gL_{12} R), \log L_{22} + \log \;( 2L_{22} R + g_{\max} + L_p + R) \Big\}$ and $\kappa = \frac{L_{22}}{\sigma}$, 
	then there exists an iteration $t\in \{0,\cdots, T\}$ such that $(\btheta_t,\balpha_{t+1})$ is an $\epsilon$--FNE of \eqref{eq: game1-cons}. 
	}


\begin{proof}
First, by  descent lemma we have

\begin{align*}
& \gl(\btheta_{t+1}) + \ql(\btheta_{t+1}) \\
\leq & \; \gl(\btheta_t) + \langle \nabla_{\btheta}\gl(\btheta_t), \btheta_{t+1} - \btheta_t\rangle + \frac{L_g}{2}\|\btheta_{t+1} - \btheta_t\|^2 + \ql(\btheta_{t+1})\\
= & \;\gl(\btheta_t) + \langle \hls, \btheta_{t+1} - \btheta_t\rangle + \frac{L_g}{2}\|\btheta_{t+1} - \btheta_t\|^2 + \ql(\btheta_{t+1})
\\ = & \; \gl(\btheta_t) + \langle \hl, \btheta_{t+1} - \btheta_t\rangle + \frac{L_g}{2}\|\btheta_{t+1} - \btheta_t\|^2 + \ql(\btheta_{t+1})
\\&  - \langle \hl - \hls, \btheta_{t+1} - \btheta_t\rangle \\ {=} & \; \gl(\btheta_{t}) + \ql(\btheta_{t})+ \min_{\btheta \in \cSt} \Big[ \langle \hl, \btheta - \btheta_t\rangle +  \frac{L_g}{2}\|\btheta - \btheta_t\|^2 + \ql(\btheta) - \ql(\btheta_{t})\Big]
\\&  - \langle \hl - \hls, \btheta_{t+1} - \btheta_t\rangle,
\end{align*}
where the last equality follows the definition of $\btheta_{t+1}$. Thus we get, 
\begin{align}
& \gl(\btheta_{t+1}) + \ql(\btheta_{t+1})
\nonumber\\ \leq & \; \gl(\btheta_{t}) + \ql(\btheta_{t})+ \min_{\btheta \in \cSt} \Big[ \langle \hl, \btheta - \btheta_t\rangle +  \frac{L_g}{2}\|\btheta - \btheta_t\|^2 + \ql(\btheta) - \ql(\btheta_{t})\Big]\nonumber
\\&  - \langle \hl - \hls, \btheta_{t+1} - \btheta_t\rangle \nonumber \\ = &  \; \gl(\btheta_{t}) + \ql(\btheta_{t})+ \frac{1}{2 L_g} 2L_g\min_{\btheta \in \cSt} \Big[ \langle \hl, \btheta - \btheta_t\rangle +  \frac{L_g}{2}\|\btheta - \btheta_t\|^2 + \ql(\btheta) - \ql(\btheta_{t})\Big] \nonumber
\\&  - \langle \hl - \hls, \btheta_{t+1} - \btheta_t\rangle
\nonumber\\ \stackrel{\circled{1}}{\leq} &  \gl(\btheta_{t}) + \ql(\btheta_{t})+{ \frac{1}{2 L_g} 2L_{11}\min_{\btheta \in \cSt} \Big[ \langle \hl, \btheta - \btheta_t\rangle +  \frac{L_{11}}{2}\|\btheta - \btheta_t\|^2 + \ql(\btheta) - \ql(\btheta_{t})\Big]}
\nonumber\\& { - \langle \hl - \hls, \btheta_{t+1} - \btheta_t\rangle }
\nonumber\\  \leq &  \; \gl(\btheta_t) + \ql(\btheta_t) - \frac{1}{2L_g}\cX(\btheta_{t},\balpha_{t+1}) + L_{12}\|\balpha_K(\btheta_t) - \balpha^*(\btheta_t)\|R,  \label{eq:LastInq}
\end{align}
where \circled{1} is due to \cite[Lemma 1]{karimi2016linear}. 
Now if we choose
\[
K_{1} \geq   2 \sqrt{8\kappa}\Bigg(2\log 2 + \log \;\left( L_gL_{12} R\right) + 2 \log \;\left(\frac{1}{\epsilon}\right) + \frac{1}{2} \log \;\left(\frac{2\Delta}{\sigma}\right)\Bigg),
\]
we have 
\[
L_{12}R\|\balpha_K(\btheta_t) - \balpha^*(\btheta_t)\| \leq \frac{\epsilon^2}{4L_g},
\]
due to Lemma~\ref{lm:projected-GD}. Combining this inequality with~\eqref{eq:LastInq} and  summing up both sides of the inequality~\eqref{eq:LastInq}, we obtain 
\begin{align*}
\sum_{t = 0}^{T -1} \left(\frac{1}{2L_g}\cX(\btheta_{t},\balpha_{t+1}) - \frac{\epsilon^2}{4L_g}\right) \leq \gl(\btheta_0) + q(\btheta_0) - (\gl(\btheta_T) + q(\btheta_T)) \leq D.
\end{align*}

As a result, by picking $T \geq \frac{4L_g D }{\epsilon^2}$, at least for one of the iterates $t \in \{1,\cdots, T\}$ we have $\cX(\btheta_{t},\balpha_{k}(\btheta_t)) \leq \epsilon^2 $. 

On the other hand, for that point $t$ from Lemma~\ref{lm:error_bound}, if we choose 
\[K_2 \geq  2 \sqrt{8\kappa}\Bigg(\log L_{22} + \log \;\left( 2L_{22} R + g_{\max} + L_p + R\right) + 2 \log \;\left(\frac{1}{\epsilon}\right) + \frac{1}{2} \log \;\left(\frac{2\Delta}{\sigma}\right)\Bigg)\]
we have $\cY(\bt) \leq \epsilon^2$. Finally setting $K =\max\{K_1, K_2\}$ will result in $\cY(\bt)\leq \epsilon^2$ and $\cX(\btheta_{t}, \balpha_{t+1})\leq \epsilon^2$. This completes the proof.

\end{proof}


\noindent\textbf{Theorem~\ref{thm: FW}.} [Formal Statement]
{\it Consider the min-max zero sum game
\[
\min_{\btheta \in \cSt}\; \max_{\balpha \in \cSa} \;\;\bigg(f(\btheta,\balpha) = h(\btheta,\balpha) -p(\balpha) + q(\btheta)\bigg),
\]
where the function $h(\btheta,\balpha)$ is concave. Define $f_{\lambda}(\btheta,\balpha) = f(\btheta,\balpha) -\frac{\lambda}{2} \|\balpha - \hat{\balpha}\|^2$ and $g_{\lambda}(\btheta) = \max\limits_{\balpha\in\mathcal{A}} h(\btheta,\balpha) -\frac{\lambda}{2} \|\balpha - \hat{\balpha}\|^2 - p(\balpha)$ for some $\hat{\balpha} \in \cSa$ . Let $D = g_{\lambda}(\btheta_0) + q(\btheta_0) - \displaystyle{\min_{\btheta\in \Theta}}\left(g_{\lambda}(\btheta) + q(\btheta)\right)$ and $L_{g_{\lambda}} = L_{11} + \frac{L_{12}^2}{\lambda}$ be the Lipschitz constant of the gradient of $g_{\lambda}$. In Algorithm~\ref{alg: alg_grad} if we set $ \eta_1 = \dfrac{1}{L_{22} + \lambda},\;\eta_2 = \frac{1}{L_{g_{\lambda}}},\; N = \sqrt{\frac{8 (L_{22} + \lambda)}{\lambda}}-1,\;\lambda = \min \{L_{22}, \frac{\epsilon}{2\sqrt{2}R}\}$ and choose $K$ and $T$ large enough such that, 
\[
T \geq N_T(\epsilon) \triangleq  \frac{8L_{g_{\lambda}} D}{\epsilon^2}, \quad
\]
and 
\[
 K \geq N_K(\epsilon) \triangleq  2 \sqrt{8\kappa}\Bigg(C + 2 \log \;\left(\frac{2}{\epsilon}\right) + \frac{1}{2} \log \;\left(\frac{2\Delta}{\lambda}\right)\Bigg),
\]
 where $C = \max \Big\{2\log 2 + \log\;(L_{g_{\lambda}}L_{12} R), \log \left(L_{22} + \lambda\right) + \log \;\left( 2\left(L_{22} + \lambda\right) R + g_{\max}^{\lambda} + L_p + R\right) \Big\}$, $\kappa = \frac{L_{22} + \lambda}{\lambda}$ and $g_{\max}^{\lambda} = \max\limits_{\balpha \in \cSa }\|\nabla_{\balpha} h(\btheta_t,\balpha)\| + \lambda R$, there exists $t \in \{0, \ldots, T\}$ such that $(\btheta_t, \balpha_{t+1})$ is an $\epsilon$-FNE of the original problem~\eqref{eq: game1-cons}. 
}
\newline
\noindent {\it Proof.}
We only need to show that when the regularized function converges to  $\epsilon$-FNE, by proper choice of $\lambda$, the converged point is also an $\epsilon$-FNE of the original game. 

It is important to notice that in the regularized function the smooth term is $h_{\lambda}(\btheta,\balpha) = h(\btheta,\balpha) - \frac{\lambda}{2} \|\balpha - \hat{\balpha}\|^2$. As a result, from Assumption~\ref{assumption: LipSmooth-uncons} we have 
\[
\|\nabla_{\balpha} h_{\lambda}(\btheta,\balpha_1)-\nabla_{\balpha} h_{\lambda}(\btheta,\balpha_{2})\| = \| \nabla_{\balpha} h(\btheta,\balpha_1)-\nabla_{\balpha} h(\btheta,\balpha_{2})  - \lambda (\balpha_1 - \balpha_2)\|\leq (L_{22}+ \lambda)\|\balpha_1-\balpha_{2}\|,
\]
where the last inequality is obtained by combing triangular inequality and Lipshitz smoothness of the function $h(.,.)$. 
Additionally, $\nabla_{\btheta} h_{\lambda}(\bar{\btheta}, \bar{\balpha}) = \nabla_{\btheta} h(\bar{\btheta}, \bar{\balpha})$.

Now, based on Definition~\ref{def:cons-approx-stationarity}, a point $(\bar{\btheta}, \bar{\balpha})$ is said to be $\epsilon$--FNE of the regularized function if $\cX_{\lambda}(\bar{\btheta}, \bar{\balpha}) \leq \epsilon^2  $ and $\cY_{\lambda}(\bar{\btheta}, \bar{\balpha})\leq \epsilon^2  $ where 
\begin{align*}
	 \cX_{\lambda}(\bar{\btheta}, \bar{\balpha}) \triangleq - 2 L_{11}\min_{ \btheta  \in \cSt}\,\,\Big[  \langle \nabla_{\btheta} h(\bar{\btheta}, \bar{\balpha}), \btheta -\bar{\btheta} \rangle + q(\btheta) - q(\bar{\btheta})  +  \frac{L_{11}}{2} ||\btheta - \bar{\btheta}||^2  \Big],    
\end{align*}
and
\begin{align*}
\cY_{\lambda}(\bar{\btheta}, \bar{\balpha}) \triangleq 2 (L_{22} + \lambda)\max_{\balpha \in \cSa}\,\,  \Big[\langle \nabla_{\balpha} h(\bar{\btheta} , \bar{\balpha})- \lambda (\bar{\balpha} - \hat{\balpha}), \balpha -\bar{\balpha} \rangle  - p(\balpha) + p(\bar{\balpha}) - \frac{(L_{22} + \lambda)}{2} ||\balpha - \bar{\balpha}||^2 \Big].
\end{align*}

For simplicity, let $\cX_{0}(\cdot,\cdot)$ and $\cY_{0}(\cdot,\cdot)$ represent the above definitions for the original function.
In the following we show that by proper choice of $\lambda$ the proposed algorithm will result in a point that $\cX_{0}(\cdot,\cdot) \leq \epsilon^2$ and $\cY_{0}(\cdot,\cdot) \leq \epsilon^2$.
To show this, we first bound the $\cY_0(\cdot,\cdot)$ by  $ \cY_{\lambda}(\cdot,\cdot)$:
\begin{align*}
&\cY_0(\bt) \\ 
= & 2L_{22} \max_{\balpha \in \cSa}\,\, \Big[	\langle \hla, \balpha -\balpha_{t+1} \rangle  - p(\balpha) + p(\balpha_{t+1}) - \frac{L_{22}}{2} ||\balpha - \balpha_{t+1}||^2 \Big]
\\ \stackrel{\circled{1}}{\leq} & 2(2L_{22} + \lambda) \max_{\balpha \in \cSa}\,\, \Big[	\langle \hla, \balpha -\balpha_{t+1} \rangle  - p(\balpha) + p(\balpha_{t+1}) - \frac{(2L_{22} + \lambda)}{2} ||\balpha - \balpha_{t+1}||^2 \Big]
\\  = & 2(2L_{22} + \lambda) \max_{\balpha \in \cSa}\,\, \Big[	\langle \hla - \lambda (\balpha_{t+1} - \hat{\balpha}) + \lambda (\balpha_{t+1} - \hat{\balpha})  , \balpha -\balpha_{t+1}\rangle  - p(\balpha) + p(\balpha_{t+1})  
\\ & \quad -\frac{2L_{22} + \lambda}{2} ||\balpha - \balpha_{t+1}||^2 \Big],
\end{align*}
where $\circled{1}$ is based on \cite[Lemma 1]{karimi2016linear}. Hence, 
\begin{align*}
&\cY_0(\bt) \\
\leq & 2\frac{2L_{22} + \lambda}{L_{22} + \lambda}(L_{22} + \lambda) \max_{\balpha \in \cSa}\,\, \Big[	\langle \hla - \lambda (\balpha_{t+1} - \hat{\balpha}) + \lambda (\balpha_{t+1} - \hat{\balpha})  , \balpha -\balpha_{t+1}\rangle  - p(\balpha)
\\ & \quad + p(\balpha_{t+1}) -\frac{2L_{22} + \lambda}{2} ||\balpha - \balpha_{t+1}||^2 \Big]
\\  \leq & 4(L_{22} + \lambda) \max_{\balpha \in \cSa}\,\, \Big[	\langle \hla - \lambda (\balpha_{t+1} - \hat{\balpha}) + \lambda (\balpha_{t+1} - \hat{\balpha})  , \balpha -\balpha_{t+1}\rangle  - p(\balpha) + p(\balpha_{t+1}) 
\\ & \quad -\frac{2L_{22} + \lambda}{2} ||\balpha - \balpha_{t+1}||^2 \Big]
\\  = & 4(L_{22} + \lambda) \max_{\balpha \in \cSa}\,\, \Big[	\langle \hla - \lambda (\balpha_{t+1} - \hat{\balpha}), \balpha -\balpha_{t+1}\rangle  - p(\balpha) + p(\balpha_{t+1})
\\ & \quad -\frac{L_{22} + \lambda}{2} ||\balpha - \balpha_{t+1}||^2  -\frac{L_{22}}{2} ||\balpha - \balpha_{t+1}||^2 + \langle \lambda (\balpha_{t+1} - \hat{\balpha})  , \balpha -\balpha_{t+1}\rangle\Big]
\\  \leq & 4(L_{22} + \lambda) \max_{\balpha \in \cSa}\,\, \Big[	\langle \hla -  \lambda (\balpha_{t+1} - \hat{\balpha}), \balpha -\balpha_{t+1}\rangle  - p(\balpha) + p(\balpha_{t+1})
\\ &  \quad -\frac{L_{22} + \lambda}{2} ||\balpha - \balpha_{t+1}||^2\Big] +  4(L_{22} + \lambda) \max_{\balpha \in \cSa} \left[-\frac{L_{22}}{2} ||\balpha - \balpha_{t+1}||^2 + \langle \lambda (\balpha_{t+1} - \hat{\balpha})  , \balpha -\balpha_{t+1}\rangle\right]
\\  \leq &  2\cY_{\lambda}(\bt) + 2 \frac{L_{22} + \lambda}{L_{22}}\lambda^2 R^2,
\end{align*}

where $\circled{1}$ is based on \cite[Lemma 1]{karimi2016linear} and the last inequality follows the definition and optimizing the quadratic term. 
As a result, by choosing $\lambda \leq \min\{{L_{22}, \frac{\epsilon}{2\sqrt{2}R}}\}\triangleq \mathcal{O}(\epsilon)$ we have, 

\begin{align*}
\cY_0(\bt)   \leq  2\cY_{\lambda}(\bt) + 2 \frac{L_{22} + \lambda}{L_{22}}\lambda^2 R^2
\leq \frac{\epsilon^2}{2} + \frac{\epsilon^2}{2} = \epsilon^2,
\end{align*}
where the last inequality comes from the fact that by running Algorithm~\ref{alg: alg_grad} with the given inputs, the regularized function has resulted in a $\frac{\epsilon}{2}$--FNE point.
Now, since $\cX(\btheta_{t},\balpha_{t+1}) $ is same for both original and regularized function, by picking $T \geq N_T(\epsilon) \triangleq  \frac{4L_{g_{\lambda}} D}{\epsilon^2} =\frac{4D}{\epsilon^2}\left(L_{11} + \frac{L_{12}^2}{\lambda}\right) \triangleq \mathcal{O}(\epsilon^{-3})$ , we conclude  $\cX_{0}(\btheta_{t},\balpha_{t+1})  \leq \epsilon^2$. This completes the proof.

\end{document}